\documentclass[12pt]{amsart}
\usepackage[english]{babel}
\usepackage{amsmath,amssymb,amsthm,latexsym,amsfonts}
\usepackage{tikz}
\theoremstyle{plain}
\newtheorem{lemma}{Lemma}
\newtheorem{theorem}[lemma]{Theorem}

\newtheorem{proposition}[lemma]{Proposition}
\newtheorem{definition}{Definition}
\newtheorem{conjecture}{Conjecture}
\newtheorem{remark}{Remark}

\begin{document}

\title{G-degree for singular manifolds}

\author{Maria Rita Casali, Paola Cristofori, Luigi Grasselli}

\address{M.R. Casali - P. Cristofori - Department of Physics, Informatics and Mathematics, University of Modena and Reggio Emilia (Italy)}
            \email{casali@unimore.it\quad paola.cristofori@unimore.it}
          
\address{L. Grasselli - Department of Sciences and Methods for Engineering, University of Modena and Reggio Emilia (Italy)}
           \email{luigi.grasselli@unimore.it}

\maketitle

\centerline{\textit{To Professor Maria Teresa Lozano on the occasion of her 70th birthday}}

\begin{abstract}
The G-degree of colored graphs is a key concept in the approach to Quantum Gravity via tensor models.
The present paper studies the  properties of the G-degree for the large class of graphs representing singular manifolds (including closed PL manifolds). \\
In particular, the complete topological  classification up to G-degree $6$ is obtained in dimension $3$, where {\it all} $4$-colored graphs represent singular manifolds.

\bigskip
\medskip

\smallskip

\end{abstract}

\section{Introduction} \label{intro}

In the last decade, the strong interaction between random tensor models, within Quantum Gravity theories, and the topology of edge-colored graphs has been deeply investigated. In fact, regular edge-colored graphs represent pseudomanifolds (see Section \ref{colored graphs}) and their {\it G-degree} plays an important r\^ole in the study of tensor invariants (see Section \ref{physics}).

The aim of this paper is to describe topological properties of the G-degree for the wide class of {\it singular manifolds} (including all closed PL manifolds), which can be  of interest as regards the so-called {\it 1/N expansion} within tensor models theory.

In this direction a key result, proved in Section \ref{general properties}, states the finite-to-one property for the G-degree of singular manifolds in any dimension. Thus, only a finite number of singular manifolds appears in each term of the 1/N expansion.

Section \ref{3-dim} is devoted to the $3$-dimensional case, where all regular $4$-colored graphs represent singular $3$-manifolds.
Hence, the complete classification for all orientable singular $3$-manifolds up to G-degree $6$ (Theorem \ref{classif 3-dim}) allows to know the topology of the - most significant - first seven elements in the 1/N expansion.

\section{Colored graphs and singular manifolds} \label{colored graphs}

A {\it singular (PL) $d$-manifold} ($d>1$) is a compact connected $d$-dimensional polyhedron admitting a simplicial triangulation where the links of vertices are closed connected $(d-1)$-manifolds,
while the links of all $h$-simplices with $h > 0$ are PL $(d-h-1)$-spheres.
Vertices whose links are not PL $(d-1)$-spheres are called {\it singular}.

\begin{remark}\label{correspondence-sing-boundary} {\em If $N$ is a singular $d$-manifold, then by deleting small open neighbourhoods of its singular vertices, a compact PL $d$-manifold $\check N$ is obtained.
Obviously $N=\check N$ iff $N$ is a closed manifold, otherwise $\check N$ has a non-empty boundary without spherical components.
Conversely, given a compact PL $d$-manifold $M$, a singular $d$-manifold $\widehat M$ can be constructed by capping off each component of $\partial M$ by a cone over it.

Note that, by restricting ourselves to the class of compact PL $d$-manifolds with no spherical boundary components,  the above correspondence is bijective and so singular $d$-manifolds and compact $d$-manifolds of this class can be associated to each other in a well-defined way.
As a consequence, although most of the following definitions and results will be formulated in terms of singular manifolds, they could be easily translated in the context of compact PL manifolds (with no spherical boundary components).}
\end{remark}

\begin{definition} \label{$d+1$-colored graph}
{\em A $(d+1)$-colored graph ($d \ge 1$) is a pair $(\Gamma,\gamma)$, where $\Gamma=(\mathcal{V}(\Gamma), \mathcal{E}(\Gamma))$ is a regular $d+1$ valent multigraph (i.e. loops are forbidden, but multiple edges are allowed)  and $\gamma$ is a {\it coloration} that is a map  $\gamma: \mathcal{E}(\Gamma) \rightarrow \Delta_d=\{0,\ldots, d\}$ which is injective on adjacent edges.}
\end{definition}

For every $\mathcal B\subseteq\Delta_d$ let $\Gamma_{\mathcal B}$ be the subgraph obtained from $(\Gamma, \gamma)$ by deleting all the edges that are not colored by the elements of $\mathcal B$. The connected components of $\Gamma_{\mathcal B}$ are called {\it ${\mathcal B}$-residues} or, if $\#\mathcal B = h$, {\it $h$-residues} of $\Gamma$ and we denote their number by $g_{\mathcal B}.$ In the following, if $\mathcal B =\{c_1,\ldots,c_h\}$, its complementary set in $\Delta_d$ will be denoted by $\hat c_1\ldots\hat c_h.$

\noindent A $d$-dimensional pseudocomplex\footnote{A $d$-pseudocomplex is a a collection of $d$-simplices whose $(d-1)$-faces are identified in pairs 
by affine homeomorphisms; actually, it generalizes the concept of simplicial complex since two $d$-simplices may intersect in a union of $(d-1)$-faces, although self-intersections are forbidden.} 
$K(\Gamma)$ associated to a $(d+1)$-colored graph $(\Gamma, \gamma)$ can be constructed in the following way:
\begin{itemize}
\item for each vertex of $\Gamma$ let us consider a $d$-simplex and label its vertices by the elements of $\Delta_d$;
\item for each pair of $c$-adjacent vertices of $\Gamma$ ($c\in\Delta_d$), the corresponding $d$-simplices are glued along their $(d-1)$-dimensional faces opposite to the $c$-labeled vertices, the gluing being determined by the identification of equally labeled vertices.
\end{itemize}

In general $K(\Gamma)$ is a {\it $d$-pseudomanifold}, that is a pure, non-branching and strongly connected pseudocomplex (\cite{Seifert-Threlfall}); $(\Gamma, \gamma)$ is said to {\it represent} $|K(\Gamma)|$, which, with a slight abuse of language, will be called a $d$-pseudomanifold, too.

Note that, as a consequence of its construction, $K(\Gamma)$ is endowed with a vertex-labeling by $\Delta_d$ that is injective on any simplex. 
Moreover, there exists a bijective correspondence between the $h$-residues of $\Gamma$ colored by any $\mathcal B\subseteq\Delta_d$ and the $(d-h)$-simplices of $K(\Gamma)$ whose vertices are labeled by $\Delta_d - \mathcal B$.

In particular, for any color $c\in\Delta_d$ each connected component of $\Gamma_{\hat c}$ is a $d$-colored graph representing  a pseudocomplex that is PL-homeomorphic to the link of a $c$-labeled vertex of $K(\Gamma)$ in its first barycentric subdivision.
Therefore, $|K(\Gamma)|$ is a singular $d$-manifold (resp. a closed $d$-manifold) iff, for each color $c\in\Delta_d$, all $\hat c$-residues of $\Gamma$ represent closed $(d-1)$-manifolds (resp. a $(d-1)$-sphere).

In virtue of the bijection described in Remark \ref{correspondence-sing-boundary}, a $(d+1)$-colored  graph $(\Gamma,\gamma)$ is said to {\it represent}
a compact PL $d$-manifold $M$ with no spherical boundary components if and only if  it represents the associated singular manifold $\widehat M$.

The following theorem extends to singular manifolds a well-known result - due to \cite{Pezzana}  - founding the combinatorial representation theory for closed PL-manifolds of arbitrary dimension via colored graphs (the so-called {\it crystallization theory}).

\begin{theorem} \label{Theorem_gem}
Any singular $d$-manifold  
admits a $(d+1)$-colored graph representing it.
\end{theorem}

\begin{proof} Let us consider a singular $d$-manifold $N$ and let $K'$ be the first barycentric subdivision of a (pseudo)triangulation $K$ of $N$.
Then each vertex of $K'$ can be labeled by the dimension of the open simplex of $K$ containing it.
As a consequence we obtain a vertex-labeled $d$-pseudocomplex triangulating $N.$  
Let $\Gamma(K')$ be the $1$-complex consisting of the barycenters of the $d$-simplices and of the edges dual to the $(d-1)$-simplices of $K'.$  $\Gamma(K')$ is a regular $(d+1)$-valent graph  that inherits, in an obvious way, a coloration on edges induced by the vertex-labeling of $K'$.
As it is easy to see, $\Gamma(K')$ represents $N$.
\end{proof}


We will denote by $G_{d+1}^{(s)}$ the set of $(d+1)$-colored graphs representing singular $d$-manifolds.
In particular, note that in dimension $3$ the above set coincides with the whole set of $4$-colored graphs.

We will call a $d$-residue of $\Gamma\in G_{d+1}^{(s)}$ {\it ordinary} if it represents $\mathbb S^{d-1}$, {\it singular} otherwise. Similarly, a color $c$ will be called {\it singular} if at least one of the $\hat c$-residues of $\Gamma$ is singular.

The following result establishes the existence of a particular set of embeddings of a bipartite (resp. non-bipartite) $(d+1)$-colored graph into orientable (resp. non-orientable) surfaces.

\begin{proposition}{\em (\cite{Gagliardi 1981})}\label{reg_emb}
Let $(\Gamma,\gamma)$ be a bipartite (resp. non-bipartite) $(d+1)$-colored graph of order $2p$. Then for each cyclic permutation $\varepsilon = (\varepsilon_0,\ldots,\varepsilon_d)$ of $\Delta_d$, up to inverse, there exists a cellular embedding, called \emph{regular}, of $(\Gamma,\gamma)$ into an orientable (resp. non-orientable) closed surface $F_{\varepsilon}(\Gamma)$ whose regions are bounded by the images of the $\{\varepsilon_j,\varepsilon_{j+1}\}$-colored cycles, for each $j \in \mathbb Z_{d+1}$.
Moreover, the genus (resp. half the genus)  $\rho_{\varepsilon} (\Gamma)$ of $F_{\varepsilon}(\Gamma)$ satisfies

\begin{equation*}
2 - 2\rho_\varepsilon(\Gamma)= \sum_{j\in \mathbb{Z}_{d+1}} g_{\varepsilon_j\varepsilon_{j+1}} + (1-d)p.
\end{equation*}

No regular embedding of $(\Gamma,\gamma)$ exists into non-orientable (resp. orientable) surfaces.
\end{proposition}

The \emph{Gurau degree} (often called {\it degree} in the tensor models literature) and the {\it regular genus} of a colored graph are defined in terms of the embeddings of Proposition \ref{reg_emb}.

\begin{definition} \label{Gurau-degree}
{\em Let $(\Gamma,\gamma)$ be a $(d+1)$-colored graph.
If $\{\varepsilon^{(1)}, \varepsilon^{(2)}, \dots , \varepsilon^{(\frac {d!} 2)}\}$ is the set of all cyclic permutations of $\Delta_d$ (up to inverse), $ \rho_{\varepsilon^{(i)}}(\Gamma)$ ($i=1, \dots \frac {d!} 2$) is called the {\it regular genus of $\Gamma$ with respect to the permutation $\varepsilon^{(i)}$}. Then, the \emph{Gurau degree} (or \emph{G-degree} for short) of $\Gamma$, denoted by  $\omega_{G}(\Gamma)$, is defined as
\begin{equation*}
 \omega_{G}(\Gamma) \ = \ \sum_{i=1}^{\frac {d!} 2} \rho_{\varepsilon^{(i)}}(\Gamma)
\end{equation*}
and the {\it regular genus} of $\Gamma$, denoted by $\rho(\Gamma)$, is defined as
\begin{equation*}
 \rho(\Gamma) \ = \ \min\, \Big\{\rho_{\varepsilon^{(i)}}(\Gamma)\ /\ i=1,\ldots,\frac {d!} 2\Big\}.
\end{equation*}}
\end{definition}

Note that, in dimension $2$, any bipartite (resp. non-bipartite) $3$-colored graph $(\Gamma,\gamma)$ represents an orientable (resp. non-orientable) surface $|K(\Gamma)|$ and $\rho(\Gamma)= \omega_G(\Gamma)$ is exactly the genus (resp. half the genus) of $|K(\Gamma)|.$
On the other hand, for $d\geq 3$, the G-degree of any $(d+1)$-colored graph (resp. the regular genus of any $(d+1)$-colored graph representing a closed PL $d$-manifold) is proved to be a non-negative {\it integer}, also in the non-bipartite case:  see \cite[Proposition 7]{Casali-Cristofori-Dartois-Grasselli} (resp. \cite[Proposition A]{Chiavacci-Pareschi}).

As a consequence of the definition of regular genus for colored graphs and of Theorem \ref{Theorem_gem}, two PL invariants for singular $d$-manifolds can be defined:

\begin{definition}\label{def_gen_degree} {\em Let $N$ be a singular $d$-manifold ($d\geq 2$).
The {\it generalized regular genus} of $N$ is defined as
\begin{equation*}
\overline{\mathcal G}(N)=\min \{\rho(\Gamma)\ | \ (\Gamma,\gamma)\mbox{ represents} \ N\}.
\end{equation*}
and the {\it Gurau degree} (or {\it G-degree}) of $N$ is defined as
\begin{equation*}
\mathcal D_G(N)=\min \{\omega_G(\Gamma)\ | \ (\Gamma,\gamma)\mbox{ represents} \ N\}.
\end{equation*}}
\end{definition}

For any $(d+1)$-colored graph $\Gamma$, \ $ \mathcal \omega_G(\Gamma) \ \ge \ \frac{d!}2 \cdot \rho(\Gamma)$ obviously holds.
Hence, for any singular $d$-manifold $N$:
$$ \mathcal D_G(N) \ \ge \ \frac{d!}2 \cdot \overline{\mathcal G}(N).$$

\begin{remark} \label{generalized vs regular genus}
{\em Note that a classical notion of {\it regular genus} for compact PL $d$-manifolds exists (\cite{Gagliardi 1981} and \cite{Gagliardi_boundary}) and makes use of the notion of (non-regular) {\it edge-colored graph with boundary} and of a suitable extension of Proposition \ref{reg_emb}. Regular genus turns out to be an important PL-manifold invariant, 
which extends to arbitrary dimension the notion of Heegaard genus (see \cite{Gagliardi 1981} for the proof of their coincidence in the closed $3$-dimensional case, \cite{Cristofori-Gagliardi-Grasselli}
and \cite{Cristofori} for the general result for compact $3$-manifolds) and which allows to obtain several classification results (see, for example: \cite{Ferri-Gagliardi Proc AMS 1982} and \cite{Ferri-Gagliardi Yokohama 1985} in arbitrary dimension,
the survey paper \cite{Casali-Cristofori-Gagliardi Complutense 2015} in dimension $4$, \cite{Casali-Gagliardi ProcAMS} and \cite{Casali Canadian} in dimension $5$).
Generalized regular genus and regular genus trivially coincide for closed PL $d$-manifolds, while their coincidence in the case of $3$-manifolds with connected boundary is proved in \cite{Cristofori-Gagliardi-Grasselli} and \cite{Cristofori}.}
\end{remark}

Another important PL invariant defined within crystallization theory takes into account the minimum order of a colored graph representing the involved
closed $d$-manifold. The following definition extends naturally this notion to singular $d$-manifolds:

\begin{definition} \label{gem-complexity}
{\em For each singular $d$-manifold $N$, its \emph{generalized gem-complexity} is the non-negative integer $\bar k(N)= p - 1$, where $2p$ is the minimum order of a $(d+1)$-colored graph representing $N$.}
\end{definition}

In the case of closed manifolds, several classification results have been obtained via gem-complexity, too: see, for example, \cite{BCrG} and \cite{Casali-Cristofori JKTR 2008} for $d=3$, \cite{Casali-Cristofori ElecJComb 2015} and \cite{Casali-Cristofori-Gagliardi Complutense 2015} for $d=4$.
\medskip

Finally, note that the above Definitions \ref{def_gen_degree} and \ref{gem-complexity} can also be referred to compact PL manifolds (with no spherical boundary components) via their associated singular manifolds.

\section{The physical meaning of the G-degree}\label{physics}

The colored tensor models theory arises as a possible approach to the study of Quantum Gravity: in some sense, its aim is to generalize to higher dimension the matrix models theory which, in dimension two, has shown to be quite useful at providing a framework for Quantum Gravity. The key generalization is the recovery of the so-called {\it 1/N expansion} in the tensor models context. In matrix models, the 1/N expansion is driven by the genera of the surfaces represented by Feynman graphs; in the higher dimensional setting of tensor models the 1/N expansion is driven by the G-degree, that equals the genus in dimension two.

We now sketch the r\^ole played by the G-degree in the 1/N expansion within colored tensor models theory: a more detailed description of these relationships between Quantum Gravity, tensor models and topology of colored graphs may be found in \cite{BGR}, \cite{Gurau-Ryan}, \cite{Gurau-book}, \cite{Casali-Cristofori-Dartois-Grasselli}.

In the framework of tensor models theory, colored graphs naturally arise as Feynman graphs encoding tensor trace invariants, built by pairwise contraction of covariant and controvariant indices in a polynomial in the tensor components.

\smallskip

If $V$ is a complex vector space of finite dimension $N$, the natural action of $GL(N)$ on $V$ extends to a natural action of $GL(N)^{\times d}$ on the tensor product $E=V^{\otimes d}$ and on its dual $E^*$.

Given a basis $\{e_i\}$ of $V$, each $T\in E$ and $\overline{T}\in E^*$ can be written as
$$ T=\sum_{i_1,\ldots,i_d=1}^N T_{i_1\ldots i_d} \ e_{i_1}\otimes \ldots \otimes e_{i_d}  \qquad \qquad \qquad \overline{T}=\sum_{i_1,\ldots,i_d=1}^N \overline{T}_{i_1\ldots i_d} \ e_{i_1}^*\otimes \ldots \otimes e_{i_d}^* ,$$
where $\{e_i^*\}$ denotes the basis of $V^*$ dual to $\{e_i\}.$

In order to obtain polynomial quantities which are invariant under the action of $GL(N)^{\times d}$ on both $E$ and $E^*$,
it is useful to build the so-called {\it trace invariants} by contracting the indices of the components of $T$ and $\overline{T},$ respecting their ordering.
In fact, any tensor invariant can be represented by a linear combinations of such invariants (see \cite{Gurau-book}).

Trace invariants can be formally written as
\begin{equation} \label{generic invariant}
B(T,\overline{T})=\sum_{\substack{i_h^{(l)}=1\\ \forall h\in\mathbb N_d, \ \forall l \in \mathbb N_{2p}}}^N \delta_{B}
\left(\prod_{l=1}^p {T}_{i_1^{(2l-1)}\ldots i_d^{(2l-1)}}\right) \left(\prod_{l=1}^p \overline{T}_{i_1^{(2l)}\ldots i_d^{(2l)}}\right),
\end{equation}
where $\delta_{B}$ denotes the product of all Kronecker deltas corresponding to the contractions of the indices.

\smallskip

Any trace invariant $B(T,\overline{T})$ of rank $d$ tensors can be encoded by a bipartite $d$-colored graph $(B,b)$ in the following way:
\begin{itemize}
\item [-] represent each $T$ appearing in formula \eqref{generic invariant} by a white vertex and each $\overline{T}$ by a black vertex;
\item [-] each time the index $i_c$ of a $T$ is contracted with the same index of $\overline{T}$, join the two corresponding vertices by a $c$-colored edge.
\end{itemize}
Note that, with respect to formula \eqref{generic invariant}, the integer $2p$ becomes the order of $B$, while each Kronecker delta in $\delta_{B}$ gives rise to a suitable colored edge of $(B,b).$

As an example let us consider the following (well-known) {\it quartic} invariant:
\begin{align}
Q_{1}(T,\overline{T})&:=\sum_{\substack{ i_1,\ldots,i_d=1 \\ j_1,\ldots,j_d=1}}^N \overline{T}_{i_1i_2\ldots i_d} T_{j_1i_2\ldots i_d}\overline{T}_{j_1j_2\ldots j_d} T_{i_1j_2\ldots j_d}\label{eq:quartm}.
\end{align}

The $d$-colored graph encoding $Q_{1}(T,\overline{T})$ is shown in Figure~\ref{fig:Qm1}. \\

\begin{figure}
 \begin{center}
  \includegraphics[scale=0.75]{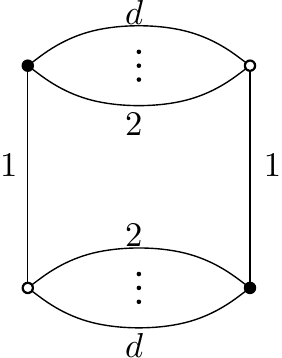}
  \caption{$d$-colored graph representing $Q_{1}(T,\overline{T})$.\label{fig:Qm1}}
 \end{center}
\end{figure}

Wick's theorem allows to obtain the following expansion of the mean value of a given trace invariant $B(T,\overline{T})$, in terms of correlations between pairs of components of $T$ and $\overline{T}$:
\begin{equation} \label{generic expansion}
\langle B(T,\overline{T})\rangle= \sum_{\substack{\sigma \in \mathcal S_p}} \left[ \sum_{\substack{i_h^{(l)}=1\\ \forall h\in\mathbb N_d, \ \forall l \in \mathbb N_{2p}}}^N \delta_{B}
\left(\prod_{r=1}^p \langle {T}_{i_1^{(2r-1)}\ldots i_d^{(2r-1)}} \, \overline{T}_{i_1^{(2\sigma(r))}\ldots i_d^{(2\sigma(r))}}\rangle \right)\right],
\end{equation}
where $\mathcal S_p$ denotes the set of all possible permutations on $\mathbb N_p.$

Each summand in the Wick expansion corresponding to a permutation $\sigma \in \mathcal S_p$ can be encoded by a (bipartite) $(d+1)$-colored graph $(\Gamma,\gamma)$ - the associated {\it Feynman graph} - starting from the $d$-colored graph $(B,b)$ representing $B(T,\overline{T})$: it suffices to add a $0$-colored edge between the white vertex associated to ${T}_{i_1^{(2r-1)}\ldots i_d^{(2r-1)}}$ and the black vertex associated to $\overline{T}_{i_1^{(2\sigma(r))}\ldots i_d^{(2\sigma(r))}}$, for each $r\in\mathbb N_p.$

Consider now a {\it $(d+1)$-dimensional colored tensor model}, i.e. a formal partition function
\begin{equation}
\mathcal{Z}[N,\{\alpha_{B}\}]:=\int_{\mbox{f}}\frac{dTd\overline{T}}{(2\pi)^{N^d}}\exp(-N^{d-1}\overline{T}\cdot T + \sum_{B}\alpha_BB(T,\overline{T})),
\end{equation}
where $T$ belongs to $(\mathbb{C}^N)^{\otimes d }$ and $\overline{T}$ to its dual, 
while $\overline{T}\cdot T=\sum_{i_1,\ldots,i_d=1}^N \overline{T}_{i_1\ldots i_d} T_{i_1\ldots i_d}.$

As shown in \cite{BGR}, its {\it free energy}  $\frac{1}{N^d}\log \mathcal{Z}[N,\{t_{B}\}]$ is the following formal series
\begin{equation} \label{1/N expansion}
\frac{1}{N^d}\log \mathcal{Z}[N,\{t_{B}\}] = \sum_{\omega_G\ge 0}N^{-\frac{2}{(d-1)!}\omega_G}F_{\omega_G}[\{t_B\}]\in \mathbb{C}[[N^{-1}, \{t_{B}\}]],
\end{equation}
\noindent where the coefficients $F_{\omega_G}[\{t_B\}]$ are generating functions\footnote{For the notions of generating series and function see \cite{[Flajolet-book]}.} of connected bipartite $(d+1)$-colored graphs with fixed G-degree $\omega_G$.

The {\it  $1/N$ expansion} of formula \eqref{1/N expansion} describes the r\^ole of colored graphs (and of their G-degree $\omega_G$) within colored tensor models theory and explains the importance of trying to understand which are the manifolds and pseudomanifolds represented by $(d+1)$-colored graphs with a given G-degree, in particular for $d=3,4.$\footnote{Indeed, in dimension 3 and 4, these graphs represent the possible states of the physical quantum space.}

\begin{remark}
{\em A parallel tensor models theory, involving {\it real} tensor variables $T\in (\mathbb{R}^N)^{\otimes d}$, has been developed, taking into account also non-bipartite colored graphs (see \cite{Witten}): this is why both bipartite and non-bipartite colored graphs will be considered in the next sections.}
\end{remark}

\section{General properties of the G-degree}\label{general properties}

An {\it $r$-dipole ($1\le r\le d$) of colors $c_1,\ldots,c_r$} of a $(d+1)$-colored graph $(\Gamma,\gamma)$ is a
subgraph of $\Gamma$ consisting of two vertices joined by $r$ edges, colored by $c_1,\ldots,c_r$, such that its vertices belong to different connected components of $\Gamma_{\hat c_1\ldots\hat c_r}$.

An $r$-dipole can be eliminated from $\Gamma$ by deleting the subgraph and welding the remaining hanging edges according to their colors; in this way another $(d+1)$-colored graph $(\Gamma^\prime,\gamma^\prime)$ is obtained. The inverse operation can be read as the addition of the dipole to $\Gamma^\prime.$

The dipole is called {\it proper} if $|K(\Gamma)|$ and $|K(\Gamma^\prime)|$ are PL-homeomorphic. This surely happens if at least one of the two connected components of $\Gamma_{\hat c_1\ldots\hat c_r}$ containing the vertices of the dipole represents a $(d-r)$-sphere (\cite[Proposition 5.3]{Gagliardi 1987}).\footnote{Note that, if $\Gamma \in G_{d+1}^{(s)}$, all $r$-dipoles are proper, for $1 < r \le d$.}

\begin{lemma}{\em (\cite{Gurau-Ryan})}\label{dipoles}
Let $(\Gamma,\gamma)$ be a $(d+1)$-colored graph and let $(\Gamma^\prime,\gamma^\prime)$ be obtained from $\Gamma$ by eliminating an $r$-dipole. Then:

$$\omega_G(\Gamma) = \frac{(d-1)!}{2}(r-1)(d-r) + \omega_G(\Gamma^\prime).$$
\end{lemma}

As a consequence of the above lemma, the elimination of an $r$-dipole does not increase the G-degree and, in particular, $1$- and $d$-dipole eliminations or additions do not change it.

\begin{definition} {\em A $(d+1)$-colored graph $(\Gamma,\gamma)$ is called {\it contracted} if for each color $c\in\Delta_d$ either $\Gamma_{\hat c}$ is connected or no connected component of $\Gamma_{\hat c}$ represents a $(d-1)$-sphere.}
\end{definition}

\begin{remark}   \label{rem: contracted}
{\em Note that $(\Gamma,\gamma)$ is contracted if it admits no proper $1$-dipoles.
Moreover, {\it any singular $d$-manifold $N$ 
can be represented by a contracted $(d+1)$-colored graph}. In fact, starting from any $(d+1)$-colored graph representing $N$ and eliminating a suitable number of proper $1$-dipoles, we obtain a contracted graph representing $N.$} 
\end{remark}

Hence, by Lemma \ref{dipoles}, it is easy to prove the following result.

\begin{proposition} \label{contracted}
For any singular $d$-manifold $N$, $\mathcal D_G(N)$
can be always realized by a contracted graph (with no $r$-dipole, $1 < r < d$).
\end{proposition}

In order to face the finiteness problem of G-degree for singular $d$-manifolds, it is useful to
recall that, for each $(d+1)$-colored graph $\Gamma$ of order $2p$, the following formula holds (\cite[Lemma 13]{Casali-Cristofori-Dartois-Grasselli}:

\begin{equation} \label{Gurau-Ryan}
 \omega_G(\Gamma) \ = \ \frac{(d-1)!}2 \Big(p + d - \sum_{i\in \Delta_d} g_{\hat i}\Big) + \sum_{i\in \Delta_d} \omega_G(\Gamma_{\hat i}),
\end{equation}
where, for each $i \in \Delta_d,$  $\omega_G(\Gamma_{\hat i})$ denotes the sum of the G-degrees of the connected components of $\Gamma_{\hat i}.$

\smallskip

For each integer $h \ge 0,$ let $\mathcal{M}^h_d$ denote the set of all singular $d$-manifolds with $h$ singular vertices. 
Obviously, for $h=0$, the set $\mathcal{M}^0_d$ coincides with the set of closed $PL$ $d$-manifolds.

\begin{theorem} \label{finiteness}
For each pair of fixed integers $\bar S\geq 0$ and $\bar R \ge d+1$, only a finite number of $(d+1)$-colored graphs $(\Gamma,\gamma)$ exists, with $\omega_G(\Gamma) = \bar S\ $ and $\ \sum_{i \in \Delta_d} g_{\hat i}= \bar R.$ 
\par \noindent
Hence, the G-degree $\mathcal D_G$ is finite-to-one within the set $\mathcal{M}^h_d$, for any $h \ge 0$.
\end{theorem}

\begin{proof}
If $\sum_{i \in \Delta_d} g_{\hat i}= \bar R$, formula \eqref{Gurau-Ryan} yields:
$$p \le \frac {2 \omega_G(\Gamma)}{(d-1)!} + (\bar R-d).$$
The first statement easily follows by observing that, for each fixed  $p\ge 1$, only a finite number of $(d+1)$-colored graphs $(\Gamma,\gamma)$ exists, with \ $\# \mathcal V(\Gamma) = 2p.$

Let us now consider $(\Gamma,\gamma)\in G_{d+1}^{(s)}$ having exactly $h$ singular $d$-residues, i.e. representing a singular $d$-manifold with $h$ singular vertices.
By Proposition \ref{contracted}, we can suppose without loss of generality $\Gamma$ to be contracted.

In this case, it is easy to see that $\ \sum_{i\in \Delta_d} g_{\hat i} = d+1-m+h\ $, where $0\leq m\leq \min\{h,d+1\}$ is the number of singular colors of $\Gamma.$ 
In virtue of the first statement, only a finite number of such $(d+1)$-colored graphs exist, with fixed G-degree $\bar S$. 
Hence, only a finite number of singular $d$-manifold with $h$ singular vertices turns out to have G-degree $\mathcal D_G$ equal to $\bar S$.
\end{proof}

\begin{remark}
{\em Note that \cite[Proposition 4.2]{Gurau-book} yields a different approach to the finiteness problem for G-degree: in fact, in that statement, the number of multiple edges is taken into account, versus the (more geometrical) number of $d$-residues considered in Theorem \ref{finiteness}.}
\end{remark}

We want now to analyze the minimum value of G-degree for non-orientable singular manifolds. With this purpose, it is useful to recall the following result, obtained in \cite{Chiavacci-Pareschi}:

\begin{lemma}\label{lemma_CP}
Let  $(\Gamma,\gamma)$ be a non-bipartite  $(d+1)$-colored graph and let $\varepsilon=(\varepsilon_0,\varepsilon_1, \dots, \varepsilon_d)$ be a cyclic permutation of $\Delta_d$. If an index $i\in \Delta_d$ exists such that:
\begin{itemize}
\item[(a)] each $\{\varepsilon_{i-1},\varepsilon_{i}, \varepsilon_{i+1}\}$-residue of $\Gamma$ is bipartite;
\item[(b)] each $\widehat{\varepsilon_{i}}$-residue of $\Gamma$ is either bipartite or non-bipartite with integer
regular genus with respect to the permutation induced by $\varepsilon$ on $\widehat{\varepsilon_{i}}$;
\end{itemize}
then    $\rho_{\varepsilon}(\Gamma)$ is an integer.
\end{lemma}

As a consequence, we obtain the following statement concerning the G-degree:

\begin{theorem}\label{G-deg non-bip}

Let $(\Gamma,\gamma)$ be a non-bipartite $(d+1)$-colored graph with the property that each 3-residue is bipartite and each $d$-residue is either bipartite or non-bipartite with integer regular genus with respect to any permutation; then $\omega_G(\Gamma) \ge \frac{d!}2.$

\noindent Hence, for $d\ge 4$:
\begin{itemize}
\item[(a)] $\omega_G(\Gamma) \ge \frac{d!}2$  for each non-bipartite graph $\Gamma$ belonging to $G_{d+1}^{(s)}$;
\item[(b)] if $\Gamma$ is a (bipartite or non-bipartite) graph belonging to $G_{d+1}^{(s)}$ and $\omega_G(\Gamma) < \frac{d!}2$,  then $ |K(\Gamma)| \cong_{PL} \mathbb S^d;$
\item[(c)] no non-orientable singular $d$-manifold $N$ exists, with $\mathcal D_G (N) < \frac{d!}2.$
 \end{itemize}
\end{theorem}

\begin{proof}
If each 3-residue of $\Gamma$ is bipartite and each $d$-residue is either bipartite or non-bipartite with integer regular genus with respect to any permutation, Lemma  \ref{lemma_CP}  ensures $\rho_{\varepsilon}(\Gamma)$ to be an integer for each cyclic permutation $\varepsilon$ of $\Delta_d$. Since $\rho_{\varepsilon}(\Gamma)=0$ can't hold because of the non-bipartiteness of $\Gamma$, $\omega_G(\Gamma) \ge \frac{d!}2$ directly follows from the definition of G-degree.

On the other hand, if $\Gamma$ belongs to $G_{d+1}^{(s)}$, each $d$-residue represents a $(d-1)$-manifold and hence, if $d\ge 4$, it is either bipartite or non-bipartite with integer
regular genus with respect to any permutation, in virtue of Lemma \ref{lemma_CP} (see also \cite[Proposition A]{Chiavacci-Pareschi});
moreover, each 3-residue of $\Gamma$ is bipartite since it represents a $2$-sphere.
Statements (a) and (c) are now trivial consequences. Statement (b) follows from statement (a), by recalling that, if $\Gamma$ is a bipartite $(d+1)$-colored graph such that $\omega_G(\Gamma) < \frac{d!}{2}$, then $|K(\Gamma)|\cong_{PL}\mathbb S^d$ (\cite[Proposition 9]{Casali-Cristofori-Dartois-Grasselli}).
\end{proof}

\begin{remark} {\em Statements (a) and (b) of the above theorem improve, for $d\ge 4$ and within the set $G_{d+1}^{(s)}$, the results of \cite[Proposition 8]{Casali-Cristofori-Dartois-Grasselli}
and \cite[Proposition 9]{Casali-Cristofori-Dartois-Grasselli} respectively.  We point out that there is no possibility of similar general improvements for $d=3$, since the minimal $4$-colored graph representing $\mathbb{RP}^2\times I$ (and also the associated singular manifold) - described in \cite{Cristofori-Mulazzani} -  has G-degree $2.$ However, as shown in Proposition \ref{G-deg3 non-bip} below, an analogue of Theorem \ref{G-deg non-bip} can be stated for particular classes of $4$-colored graphs and $3$-manifolds.}
\end{remark}

\section{The $3$-dimensional case}\label{3-dim}

As already pointed out, in dimension $3$ all 4-colored graphs represent singular $3$-manifolds. Hence, the study of the properties concerning singular $3$-manifolds 
covers all the elements of the 1/N expansion of formula \eqref{1/N expansion}, in the $3$-dimensional tensor models framework.

\smallskip

For $3$-dimensional closed manifolds, it is known that the G-degree coincides with the gem-complexity (\cite[Theorem 16]{Casali-Cristofori-Dartois-Grasselli}). 
In the case of singular 3-manifolds with only one singular vertex, 
the relation between the G-degree and generalized gem-complexity involves also the genus of the singularity, 
i.e. the genus of the boundary of the associated compact $3$-manifold. 

For sake of notational simplicity, the following proposition as well as all results of the present section will be stated in terms of $3$-manifolds with boundary, even if - in virtue of Remark \ref{correspondence-sing-boundary} - 
they may be obviously  re-stated  by involving the associated singular $3$-manifolds.

\begin{proposition}\label{G-degree vs gem-complexity (n=3)}
For each 3-manifold $M^3$ with connected non-spherical boundary,
$$  \mathcal D_G(M^3) = \bar k (M^3) +  g^{\partial}   \ \ \ \ \
\text{where} \ \ g^{\partial} =  \begin{cases}  genus (\partial M^3) \ & \ \text{if} \  \partial M^3 \ \text{is orientable} \\
                                               \frac 1 2 genus (\partial M^3) \ & \ \text{if} \  \partial M^3 \ \text{is non-orientable}
                                 \end{cases}$$
\end{proposition}

\begin{proof} Let $(\Gamma,\gamma)$ be a $4$-colored graph realizing $\mathcal D_G(M^3).$ By Proposition \ref{contracted}, we can suppose $\Gamma$ to be contracted and hence,
since $\partial M^3$ is connected, $g_{\hat i}= 1$ for all $i\in \Delta_3.$

Therefore, by applying formula \eqref{Gurau-Ryan} (with $d=3$), we have

\begin{equation*}\omega_G(\Gamma) = p - 1 + \sum_{i\in \Delta_3} \omega_G(\Gamma_{\hat i}) = p - 1 + g^{\partial}\end{equation*}

\noindent where the last equality comes from recalling that the G-degree of a bipartite (resp. non-bipartite) $3$-colored graph equals the genus (resp. half the genus) of the represented surface.

It is now clear that $\Gamma$ must realize the generalized gem-complexity of $M^3$, too, hence the claim is proved.
\end{proof}

In the case of a compact 3-manifold $M^3$ with $h\ge 2 $  (non-spherical) boundary components,
the total number of $3$-residues in a contracted (vertex-)minimal graph representing $M^3$ is not a priori determined.
However, in all known situations, generalized gem-complexity is realized by graphs with the minimum possible number of $3$-residues, i.e. $\max \{4, h\}$.
This fact suggests the following

\begin{conjecture}\label{conj G-degree vs gem-complexity (n=3)}
For each compact 3-manifold $M^3$ with $h \ge 2$ non-spherical boundary components $\ ^{\partial} M_1, \dots, ^{\partial} M_h$,
$$  \mathcal D_G(M^3) =  \bar k (M^3) +  \sum_{i=1}^h g^{\partial}_i  \ \ \ \ \ \text{where} \ \ g^{\partial}_i = \begin{cases} genus (^{\partial} M_i) \ & \ \text{if} \ \ ^{\partial} M_i \ \text{is orientable} \\
                                       \frac 1 2 genus (^{\partial} M_i) \ & \ \text{if} \  \ ^{\partial} M_i \ \text{is non-orientable} \end{cases}$$
\end{conjecture}

In  \cite{Casali-Cristofori-Dartois-Grasselli} the classification is given of $4$-colored graphs representing closed $3$-manifolds up to G-degree $5$.
Proposition 9 of the same paper, for $d=3$, proves that any bipartite $4$-colored graph with G-degree less than $3$ does represent the $3$-sphere.

The following proposition extends the result to a wider class of graphs.

\begin{proposition}\label{G-deg3 non-bip}
Let $(\Gamma,\gamma)$ be a $4$-colored graph with the property that each $3$-residue is bipartite. Then, either $\omega_G(\Gamma) \ge 3$ or $\Gamma$ represents the $3$-sphere.

\noindent Therefore, no non-orientable $3$-manifold $M$ with empty or orientable boundary exists, such that $\mathcal D_G (M) < 3.$
\end{proposition}

\begin{proof}
The proof, as in the case of Theorem \ref{G-deg non-bip}, is a direct consequence of Lemma \ref{lemma_CP} and \cite[Proposition 9]{Casali-Cristofori-Dartois-Grasselli}. 
\end{proof}

In the particular case of orientable manifolds with no spherical boundary components, the classification can be extended up to G-degree $6$.
In the following propositions we will denote by $\mathbb Y^3_g$ the $3$-dimensional orientable genus $g$ handlebody, by $\mathbb D^2_2$ the disk with two holes and by $M_S$ the Seifert manifold with base $\mathbb D^2$ and Seifert parameter $(2,1),(2,1)$.

\begin{proposition}
Let $(\Gamma,\gamma)$ be a $4$-colored bipartite graph representing an orientable $3$-manifold $M$ with non-empty boundary and no spherical boundary components.
Then:
  \begin{eqnarray*}
  &\omega_G(\Gamma)& = 3\ \Longrightarrow\ M=\mathbb Y^3_1
 \\
 &\omega_G(\Gamma)& \in \{4,5\}  \ \Longrightarrow\
M \in \{\mathbb Y^3_1, \ \mathbb T^2\times I\}
\\
&\omega_G(\Gamma)& = 6
 \ \Longrightarrow\
M \in \{\mathbb Y^3_1, \ \mathbb T^2\times I,\ \mathbb D^2_2\times\mathbb S^1,\ \mathbb Y^3_1\#\mathbb Y^3_1,\ M_S, \ 
(\mathbb S^1\times\mathbb S^2)\#\mathbb Y^3_1,\ \mathbb{RP}^3\#\mathbb Y^3_1,\ \mathbb Y^3_2\}
  \end{eqnarray*}
  \end{proposition}

 \begin{proof}
Let $h > 0$ be the number of connected (non-spherical) boundary components of $M$.
By Lemma \ref{dipoles} and Remark \ref{rem: contracted}  we can suppose that $\Gamma$ is contracted; therefore, as in the proof of Theorem \ref{finiteness}, we have $\sum_{i\in \Delta_3} g_{\hat i}\leq h+3.$

On the other hand, with the notations of Conjecture \ref{conj G-degree vs gem-complexity (n=3)}, we have
 $$\sum_{i\in \Delta_3} \omega_G(\Gamma_{\hat i}) = \sum_{j=1}^h g^{\partial}_j.$$

Then, formula \eqref{Gurau-Ryan} for $d=3$ gives
$$\omega_G(\Gamma) \ = \ p + 3 - \sum_{i\in \Delta_3} g_{\hat i} + \sum_{i\in \Delta_3} \omega_G(\Gamma_{\hat i}) \ge \ p - h + \sum_{j=1}^h g^{\partial}_j = p + \sum_{j=1}^h (g^{\partial}_j - 1) \ge p .$$

Note that Lemma \ref{dipoles} ensures that there exists a $4$-colored graph $\Gamma^\prime$ still representing $M$, such that $\Gamma^\prime$ is contracted and without $2$-dipoles, $\omega_G(\Gamma^\prime)\le\omega_G(\Gamma)$ and $p'\le p$, where $2p'$ is the order of $\Gamma^\prime$.

Now the statement is proved by taking into account the catalogues of all contracted bipartite $4$-colored graphs with order up to $12$ and without $2$-dipoles, that are presented in \cite{Cristofori-Mulazzani} and \cite{Cristofori-Fomynikh-Mulazzani-Tarkaev}.
More precisely, there are no such graphs up to order $4$ and the only graphs (of order $6, 8$ and $10$) having G-degree less or equal $5$ represent either $\mathbb T^2\times I$ or the genus one handlebody.

Moreover, irreducible and boundary-irreducible compact $3$-manifolds with toric boundary represented by graphs of G-degree $6$ and order $10$ and $12$ are among those classified in \cite{Cristofori-Fomynikh-Mulazzani-Tarkaev}, while reducible manifolds with toric boundary are easily identified by making use of the program {\it 3-Manifold Recognizer}\footnote{{\it 3-Manifold Recognizer} is a computer program for studying $3$-manifolds written by V. Tarkaev according to an algorithm elaborated by S. Matveev and other members of the Chelyabinsk topology group  (\cite{Recognizer}).}.
In this way all graphs of G-degree $6$ and order up to $12$ have been considered, except five graphs $\{\Gamma_1,\ldots,\Gamma_5\}$ of order $10$ all having only one singular $3$-residue representing the closed orientable surface of genus $2.$
However, an easy computation yields $\rho(\Gamma_i) =2$ for each $1\leq i\leq 5$ and therefore $\overline{\mathcal G}(M_i) \leq 2$ for each manifold with boundary $M_i$ represented by $\Gamma_i.$

On the other hand, by a well-known result of \cite{Chiavacci-Pareschi}, we have $\ \mathcal G(M_i)\geq\mathcal G(\partial M_i)=2$, where $\mathcal G(M_i)$ denotes the {\it regular genus} of $M_i$ (see Remark \ref{generalized vs regular genus}).
Since the generalized regular genus and the regular genus coincide for $3$-manifolds with connected boundary (see again Remark \ref{generalized vs regular genus}) and, hence, $\mathcal G(M_i) = \mathcal G(\partial M_i) = 2$, $M_i$ is proved to represent the $3$-dimensional genus $2$ handlebody by the main result of \cite{Casali UMI 1990}.
\end{proof}

Formula \eqref{Gurau-Ryan}, together with the results described in the proof of the above proposition, allows to state the following classification according to G-degree.

\begin{theorem}\label{classif 3-dim} Let $M$ be an orientable $3$-manifold with non-empty boundary and no spherical boundary components. Then:
\begin{itemize}
 \item [] $\mathcal D_G(M) > 2$\ \ and $\ \ \mathcal D_G(M) \neq 5;$
 \item [] $\mathcal D_G(M) \ = \ 3 \ \ \ \ \Longleftrightarrow  \ \ \ \ M =\mathbb Y^3_1;$
 \item [] $\mathcal D_G(M) \ = \ 4 \ \ \ \ \Longleftrightarrow  \ \ \ \ M =\mathbb T^2\times I;$
 \item [] $\mathcal D_G(M) \ = \ 6 \ \ \ \ \Longleftrightarrow  \ \ \ \ M \in \{\mathbb D^2_2\times\mathbb S^1,\ \mathbb Y^3_1\#\mathbb Y^3_1,\
M_S, \  (\mathbb S^1\times\mathbb S^2)\#\mathbb Y^3_1,\ \mathbb{RP}^3\#\mathbb Y^3_1,\ \mathbb Y^3_2\}.$
\end{itemize}
\end{theorem}

\medskip

\textbf{Acknowledgements:}  This work was supported by
GNSAGA of INDAM and by the University of Modena and Reggio Emilia, projects:  {\it ``Colored graphs representing pseudomanifolds: an 
interaction with random geometry and physics"} and {\it ``Applicazioni della Teoria dei Grafi nelle Scienze, nell'Industria e nella Societ\`a"}.


\medskip


\begin{thebibliography}{}

\bibitem{BCrG}
Bandieri, P., Cristofori, P., Gagliardi, C.:  Nonorientable 3-manifolds admitting coloured triangulations with at most 30 tetrahedra, J. Knot Theory Ramifications {\bf 18}, 381-395 (2009).

\bibitem{BGR}
Bonzom, V., Gurau, R., Rivasseau, V.: Random tensor models in the large N limit: Uncoloring the colored tensor models, Phys. Rev. D {\bf 85}, 084037 (2012). 

\bibitem{Casali UMI 1990}
Casali, M.R.: Una caratterizzazione dei corpi di manici $3$-dimensionali, Boll. Un. Mat. Ital. {\bf 4-B}, 517-539 (1990).

\bibitem{Casali Canadian}
Casali, M.R.: Classifying PL 5-manifolds by regular genus: the boundary  case,  Canadian J. Math. {\bf 49}, 193-211 (1997).

\bibitem{Casali-Cristofori JKTR 2008}
Casali, M.R., Cristofori, P.: A catalogue of orientable 3-manifolds triangulated by $30$ coloured tetrahedra, J. Knot Theory Ramifications {\bf 17}, 1-23 (2008).

\bibitem{Casali-Cristofori ElecJComb 2015}
Casali, M.R., Cristofori, P.: Cataloguing PL 4-manifolds by gem-complexity, Electr. J. Comb. {\bf 22} (4),  \#P4.25 (2015).


\bibitem{Casali-Cristofori-Gagliardi Complutense 2015}
Casali, M.R., Cristofori, P., Gagliardi, C.: Classifying PL 4-manifolds via crystallizations: results and open problems, in: ``A Mathematical Tribute to Professor Jos\'e Mar\'ia Montesinos Amilibia", Universidad Complutense Madrid (2016). [ISBN: 978-84-608-1684-3]

\bibitem{Casali-Cristofori-Dartois-Grasselli}
Casali, M.R., Cristofori, P., Dartois, S., Grasselli, L.: Topology in colored tensor models  via crystallization theory, arXiv: 1704.02800.

\bibitem{Casali-Gagliardi ProcAMS}
Casali, M.R., Gagliardi, C.: Classifying PL 5-manifolds up to regular genus seven,  Proc. Amer. Math. Soc. {\bf 120},  275-283 (1994).

\bibitem{Chiavacci-Pareschi}
Chiavacci, R., Pareschi, G.: Some bounds for the regular genus of closed PL manifolds, Discrete Math. {\bf 82}, 165-180 (1990).

\bibitem{Cristofori}
Cristofori, P.: Heegaard and regular genus agree for compact 3-manifolds,  Cahiers Topologie Geom. Differentielle Categ. {\bf 39}, 221-235 (1998).

\bibitem{Cristofori-Fomynikh-Mulazzani-Tarkaev}
Cristofori, P., Fomynikh, E., Mulazzani, M., Tarkaev, V.: $4$-colored graphs and knot/link complements, Results Math. {\bf 72} (1), 471-490 (2017).   

\bibitem{Cristofori-Gagliardi-Grasselli}
Cristofori, P., Gagliardi, C., Grasselli, L.: Heegaard and regular genus of 3-manifolds with boundary, Rev. Mat. Univ. Complut. Madrid {\bf 8},  379-398 (1995).

\bibitem{Cristofori-Mulazzani}
Cristofori, P., Mulazzani, M,: Compact 3-manifolds via 4-colored graphs, RACSAM {\bf 110} (2), 395-416 (2015).

\bibitem{Ferri-Gagliardi Proc AMS 1982}
Ferri, M., Gagliardi, C.: The only genus zero $n$-manifold is $\mathbb S^n$, Proc. Amer. Math. Soc. {\bf 85}, 638-642 (1982).

\bibitem{Ferri-Gagliardi Yokohama 1985} Ferri, M., Gagliardi, C.: A characterization of punctured $n$-spheres. Yokohama Math. J. {\bf 33}, 29-38 (1985).

\bibitem{[Flajolet-book]}
Flajolet, P., Sedgewick, R.: Analytic Combinatorics, Cambridge University Press, New York, (2009).

\bibitem{Gagliardi 1981} Gagliardi, C.: Extending the concept of genus to dimension $n$. Proc. Amer. Math. Soc. {\bf 81}, 473-481 (1981).

\bibitem{Gagliardi 1987} Gagliardi, C.: On a class of $3$-dimensional polyhedra. Ann. Univ. Ferrara {\bf 33}, 51-88 (1987).

\bibitem{Gagliardi_boundary}
Gagliardi, C.: Regular genus: the boundary case, Geom. Dedicata {\bf 22}, 261-281 (1987).

\bibitem{Gurau-book} Gurau, R.: Random Tensors, Oxford University Press, New York, (2016).

\bibitem{Gurau-Ryan}
Gurau, R., Ryan, J.P.: Colored Tensor Models - a review, SIGMA {\bf 8} 020 (2012).  


\bibitem{Pezzana} Pezzana, M.: Sulla struttura topologica delle variet\`a compatte, Atti Sem.
Mat. Fis. Univ. Modena {\bf 23}, 269-277 (1974).

\bibitem{Recognizer} Matveev, S., Tarkaev, V., et al.: 3-Manifolds Recognizer, April 2016. \ Available at http://www.matlas.math.csu.ru/?page=recognizer

\bibitem{Seifert-Threlfall}
Seifert, H., Threlfall, W.: A textbook of topology, Academic Press, New York, (1980).

\bibitem{Witten}
Witten, E.: An SYK-Like Model Without Disorder, arXiv:1610.09758v2.  

\end{thebibliography}
\end{document}